\documentclass[11pt]{article}

\usepackage[margin=1in]{geometry}
\usepackage{amsmath,amssymb,amsthm}
\usepackage{mathtools}
\usepackage[english]{babel}
\usepackage{algorithm}
\usepackage{algpseudocode}
\usepackage{cite}
\usepackage{booktabs}
\usepackage{pdflscape}
\usepackage{longtable}
\usepackage{graphicx}
\usepackage[bookmarks,bookmarksnumbered]{hyperref}
\hypersetup{colorlinks = true,linkcolor = blue,anchorcolor =red,citecolor = blue,filecolor = red,urlcolor = red,
            pdfauthor=author}

\newtheorem{theorem}{Theorem}[section]
\newtheorem{lemma}[theorem]{Lemma}
\newtheorem{proposition}[theorem]{Proposition}

\theoremstyle{definition}
\newtheorem{assumption}{Assumption}

\theoremstyle{remark}
\newtheorem{remark}{Remark}

\newcommand{\R}{\mathbb{R}}
\newcommand{\norm}[1]{\left\lVert #1 \right\rVert}

\title{A Wolfe-Type Spectral Conjugate Gradient Method for Nonsmooth Convex Optimization Problems}
\author{Jauny\textsuperscript{a} and Gourav Kumar \textsuperscript{b} \\ \textsuperscript{a}Department of Mathematics, SRM Institute of Science and Techology, \\Kattankulathur, Chengalpattu District, 603203, Tamil Nadu, India\\ \texttt{jaunys@srmist.edu.in}\\\textsuperscript{b}Department of Mathematics and Computing, National Institute of Technology\\ Kurukshetra 136119, Haryana, India}
\date{}

\begin{document}
\maketitle

\begin{abstract}
This paper proposes a Wolfe-type spectral conjugate gradient method for nonsmooth convex optimization, built on the Moreau-Yosida regularization of the objective function. The method combines a safeguarded spectral parameter with a Dai-Kou-type conjugate parameter, and uses a Wolfe-type line search compatible with the inexact gradients that the regularization produces. We establish global convergence of the method, together with an R-linear convergence rate under an additional strong-convexity assumption. The method is evaluated on standard nonsmooth optimization benchmarks and on large-scale problems, and compared against several existing conjugate gradient and bundle-type methods. The results show that the proposed method performs competitively overall, matching or outperforming existing methods on most problems tested, while a few specific limitations of the current implementation are also identified and discussed.
\end{abstract}


\section{Introduction}

Conjugate gradient (CG) methods are among the most efficient methods for large-scale smooth unconstrained optimization problems, owing to their iteration simplicity, low memory requirements, and strong convergence theory. CG methods address the unconstrained optimization problem
\begin{equation}\label{org_prob}
\min_{x \in \R^n} f(x), 
\end{equation}
where $f : \R^n \to \R$ is smooth. The iteration formula of the classical nonlinear CG method is
\begin{equation}\label{iter_seq_smooth}
x_{k+1} = x_k + \alpha_k d_k, 
\end{equation}
where $\alpha_k > 0$ is the step length and $d_k$ is the search direction which is defined as follows:
\begin{equation}\label{dir_formula_smooth}
d_k =
\begin{cases}
-g_0, & k = 0,\\
-g_k + \beta_k d_{k-1}, & k \ge 1,
\end{cases} 
\end{equation}
with $g_k = \nabla f(x_k)$ and $\beta_k$ the conjugate parameter. Numerous choices of $\beta_k$ have been proposed, giving rise to well-known CG variants such as Fletcher-Reeves (FR) \cite{fletcherreeves1964}, Hestenes-Stiefel (HS) \cite{hestenesstiefel1952}, Polak-Ribi\`ere-Polyak (PRP) \cite{polakribiere1969,polyak1969}, Dai-Yuan (DY) \cite{daiyuan1999}, Liu-Storey (LS) \cite{liustorey1991}, and others. A parallel line of work, initiated by Barzilai and Borwein (BB) \cite{barzilaiborwein1988}, replaces the unit coefficient of $g_k$ with a spectral parameter $\theta_k$ obtained from a two-point secant approximation of the Hessian (or its inverse):
\begin{equation}\label{thets_k_smooth}
\theta_k = \frac{s_{k-1}^\top s_{k-1}}{s_{k-1}^\top y_{k-1}} \quad \text{or} \quad \theta_k = \frac{s_{k-1}^\top y_{k-1}}{y_{k-1}^\top y_{k-1}},
\end{equation}
where $s_{k-1} = x_k - x_{k-1}$ and $y_{k-1} = g_k - g_{k-1}$.

Building on the BB spectral gradient method and the nonlinear CG framework, Birgin and Mart\'inez (2001) \cite{birginmartinez2001} proposed the spectral conjugate gradient (SCG) method, in which the search direction is updated as
\begin{equation}\label{SCG_dir_smooth}
d_{k+1} =
\begin{cases}
-g_0, & k = 0,\\
-\theta_k g_{k+1} + \beta_k d_k, & k \ge 1.
\end{cases} 
\end{equation}
By combining spectral scaling of the gradient with the conjugate memory term, SCG inherits the advantages of both approaches. This hybrid framework serves as the foundation of the method proposed in this paper. The most effective modern spectral CG variants for smooth optimization combine three key components: (i) a bounded spectral parameter which is consistent with quasi-Newton secant approximations, (ii) a conjugate parameter guaranteeing the sufficient descent property $g_k^\top d_k \le -c\norm{g_k}^2$ (independent of the line search), and (iii) a Wolfe-type line search:
\begin{align}
f(x_k + \alpha_k d_k) &\le f(x_k) + \delta \alpha_k g_k^\top d_k, \label{wolfe_con_smooth1}\\
g(x_k + \alpha_k d_k)^\top d_k &\ge \sigma g_k^\top d_k, \label{wolfe_con_smooth2}
\end{align}
with $0 < \delta < \sigma < 1$,  which provides the curvature condition required to establish global convergence.

Despite the remarkable progress achieved for smooth optimization, the nonsmooth unconstrained setting poses fundamentally different challenges. Specifically, we consider the following optimization problem: 
\begin{equation}\label{nonsmooth_prob}
\min_{x \in \R^n} f(x), 
\end{equation}
where $f : \R^n \to \R$ is convex, locally Lipschitz continuous, and not necessarily differentiable. A prominent approach for adapting CG methods to \eqref{nonsmooth_prob} is to work with the Moreau-Yosida regularization (MYR) of $f$, which replaces \eqref{nonsmooth_prob} by the equivalent smooth problem:
\begin{equation}\label{MYR_min_problem}
\min_{x \in \R^n} F(x), 
\end{equation}
where
\begin{equation}\label{F_x}
F(x) = \min_{z \in \R^n} \left\{ f(z) + \frac{1}{2\lambda}\norm{z-x}^2 \right\}, \qquad \lambda > 0.
\end{equation}
The function $F$ is convex and differentiable, with
\begin{equation}\label{grad_F}
g(x) =\nabla F(x) = \frac{x - p(x)}{\lambda}, 
\end{equation}
where
\begin{equation}\label{exact_p}
p(x) = \arg\min_{z \in \R^n} \left\{ f(z) + \frac{1}{2\lambda}\norm{z-x}^2 \right\}, 
\end{equation}
which is generally unavailable in closed form. Practical algorithms therefore work with an $\varepsilon$-inexact gradient $g^a(x,\varepsilon)$ satisfying $\norm{g^a(x,\varepsilon) - \nabla F(x)} \le \sqrt{2\varepsilon/\lambda}$, with $\varepsilon_k \downarrow 0$ along the iterations (see \cite{correa1993}).

The variant of CG methods to solve \eqref{nonsmooth_prob} via problem \eqref{MYR_min_problem} share two limitations that motivate the present work: none uses a spectral parameter within the search direction, and all rely on an Armijo-type line search condition. Yuan, Wei, and Li \cite{yuan2014} proposed a three-term Polak-Ribi\`ere-Polyak-type method using a trust-region-style denominator for sufficient descent, with a nonmonotone Armijo line search. Woldu, Zhang, Zhang, and Fissuh \cite{woldu2020} used a nonnegative hybrid Dai-Yuan/Hestenes-Stiefel parameter under a monotone Armijo line search. Abdollahi and Fatemi \cite{abdollahi2021} constructed the conjugate parameter via a small auxiliary optimization problem, again with a monotone Armijo line search. Earlier, Yuan and Wei \cite{yuan2012} introduced BB-NS, a purely spectral gradient algorithm in which the BB step sets only the initial trial step length for a nonmonotone Armijo line search, rather than persisting as a spectral scaling within the direction itself. All four require only convexity of $f$ and $\varepsilon_k\to0$.

One method that combines spectral scaling, a conjugate memory term, and a Wolfe-type line search for a nonsmooth objective is the Spectral Conjugate Subgradient (SCS) method of Loreto, Humphries, Raghavan, Wu, and Kwak \cite{loreto2025}, which works directly with Clarke subgradients without MYR smoothing. Its descent property rests on a post hoc angle test that discards the conjugate term and restarts as pure spectral descent whenever the test fails, and its convergence proof covers only the case $\beta=0$, under a boundedness hypothesis the authors themselves note ``may not hold'' once the conjugate term is active. A different route, the semismooth conjugate gradient method of Bethke, Griewank, and Walther \cite{bethke2024}, exploits semismoothness of $f$ directly (without MYR step, spectral parameter, or Wolfe search) and proves convergence to Clarke stationary points for nonconvex nonsmooth $f$, at higher per-iteration cost.

This leaves a precise gap: no existing method combines a spectral parameter, a conjugate memory term, and a curvature-controlling Wolfe-type line search under a proved global convergence guarantee, and SCS -- the only method to attempt this combination -- falls short of such a proof. The central question of this paper is whether the descent and line-search machinery that makes the strongest smooth-case spectral CG methods \cite{faramarzi2019,wang2015} rigorous can be carried over to the nonsmooth setting without sacrificing the spectral term or the two-sided line search. We answer affirmatively by working on the Moreau-Yosida smoothed surrogate $F$ rather than directly on the discontinuous subgradient field, which lets us adapt the sufficient-descent and curvature-control arguments of the smooth theory instead of relying on a heuristic restart.

The contributions of this paper are:
\begin{enumerate}
\item[(i)] A new spectral parameter $\theta_k$, quasi-Newton-consistent and safeguarded, and a conjugate parameter $\beta_k$ of Dai--Kou type, whose combination in direction \eqref{dir_nonsmooth_scg} is proved (Lemma~\ref{lem:descent} below) to satisfy sufficient descent by construction, with no heuristic restart, in contrast to SCS \cite{loreto2025}.
\item[(ii)] An R-linear convergence rate (Theorem~\ref{thm:linear} below) for the function values and iterates when $f$ is, in addition, strongly convex, obtained by combining the sufficient-descent and step-length machinery above with the fact that the MYR of a strongly convex function is itself strongly convex.
\item[(iii)]  We conduct an extensive numerical evaluation in Section 4. We test the method on the standard nonsmooth test set of Luk\v{s}an and Vl\v{c}ek \cite{luksanvlcek2000}, on which it converges to the known optimum on 22 of 23 implementable problems (Section 4.1). We compare it, under controlled, same-implementation conditions, against a reproduction of the Moreau-Yosida-based NCG method of Abdollahi and Fatemi \cite{abdollahi2021}, together with further literature comparisons where reproduction was not possible (Section 4.2). We further test the method at scale, up to $100{,}000$ variables, on the scalable problems of Tang, Ouyang, Pham, and Yuan \cite{tang2025}, including a controlled comparison against three existing large-scale methods (Section 4.3). We report favorable and unfavorable outcomes alike, including specific cases where our method is outperformed and the mechanism responsible, and a discrepancy we found between a prior paper's published results and the true optimal values on several of its own test problems.
\end{enumerate}

The rest of the paper is organized as follows. Section \ref{section_prelim} recalls the properties of the Moreau-Yosida regularization and its inexact surrogate used throughout. Section \ref{sec_method} develops the proposed method itself and establishes its sufficient-descent property, global convergence, and R-linear rate under strong convexity. Section 4 reports the numerical evaluation, from the standard Luk\v{s}an--Vl\v{c}ek test set and a controlled comparison against existing methods through to a large-scale study, up to $100{,}000$ variables, against three further existing large-scale methods.

\section{Preliminaries} \label{section_prelim}

Some properties of $F$ used below can be found in \cite{baillonhaddad1977,correa1993,hiriart1993,calamai1987,qi1993}. We collect the ones needed here.

\begin{enumerate}
\item[(i)] $F$ is finite-valued, convex, and differentiable. Its gradient mapping $g: \R^n \to \R^n$, $g(x) = \nabla F(x)$, is globally Lipschitz continuous with constant $1/\lambda$:
\begin{equation}\label{g_lips_cond}
\norm{g(x) - g(y)} \le \frac{1}{\lambda}\norm{x-y}, \qquad \forall x,y \in \R^n. 
\end{equation}
\item[(ii)] $x$ solves \eqref{nonsmooth_prob} $\iff$ $x$ solves \eqref{MYR_min_problem} $\iff$ $g(x) = 0$ $\iff$ $p(x) = x$.
\item[(iii)] Since $F$ is convex and differentiable with $g$ globally $\tfrac{1}{\lambda}$-Lipschitz continuous (see \eqref{g_lips_cond}), then
\begin{equation}\label{g_corcive}
(g(x)-g(y))^\top(x-y) \ge \lambda\|g(x)-g(y)\|^2, \qquad \forall\, x,y\in\mathbb{R}^n. 
\end{equation}
\end{enumerate}

Once $p(x)$ is known, $F$ can be computed as
\begin{equation}\label{F_val_px}
F(x) = f(p(x)) + \frac{1}{2\lambda}\norm{p(x)-x}^2, 
\end{equation}
and hence the gradient of $ F$ will be given by \eqref{grad_F}. Since $p(x)$ is often unavailable in closed form, fortunately, for each $x\in\R^n$ and $\varepsilon>0$ there exists an approximate minimizer $p^a(x,\varepsilon)$ such that
\begin{equation}\label{approx_px_cond}
f(p^a(x,\varepsilon)) + \frac{1}{2\lambda}\norm{p^a(x,\varepsilon)-x}^2 \le F(x) + \varepsilon. 
\end{equation}

Now, using the vector $p^{a}(x,\varepsilon)$, it is possible to approximate $F(x)$ and $g(x)$ by
\begin{equation}\label{aprrox_F}
    F^{a}(x,\varepsilon) = f(p^{a}(x,\varepsilon)) + \frac{1}{2\lambda}\|p^{a}(x,\varepsilon)-x\|^2,
\end{equation}
and
\begin{equation}\label{approx_g}
    g^{a}(x,\varepsilon) = \frac{x - p^{a}(x,\varepsilon)}{\lambda}, 
\end{equation}
respectively.


 The next result quantifies how closely $F^a$ and $g^a$ approximate $F$ and $g$, in terms of the tolerance $\varepsilon$.

\begin{proposition}\label{prop:myr}
Let $p^a(x,\varepsilon)$ satisfies \eqref{approx_px_cond}, and let $F^a(x,\varepsilon)$ and $g^a(x,\varepsilon)$ be as in \eqref{aprrox_F} and \eqref{approx_g}, respectively. For arbitrary $\varepsilon,\varepsilon_1,\varepsilon_2 > 0$ and $\varepsilon_3 = \max\{\varepsilon_1,\varepsilon_2\}$,
\begin{align}
F(x) &\le F^a(x,\varepsilon) \le F(x) + \varepsilon, \label{eps_F}\\
\norm{p^a(x,\varepsilon) - p(x)} &\le \sqrt{2\lambda \varepsilon}, \label{eps_p}\\
\norm{g^a(x,\varepsilon) - g(x)} &\le \sqrt{\frac{2\varepsilon}{\lambda}}, \label{eps_g}\\
\norm{g^a(x,\varepsilon_1) - g^a(y,\varepsilon_2)} &\le \frac{1}{\lambda}\norm{x-y} + \sqrt{\frac{8\varepsilon_3}{\lambda}}. \label{approx_g_lips}
\end{align}
\end{proposition}

Throughout the rest of the paper, $g^a(x_k,\varepsilon_k)$ will be abbreviated as $\tilde g_k$, which denotes the $\varepsilon_k$-inexact MYR gradient.

\section{The Proposed Method}\label{sec_method}

In this section, we develop a method for solving the optimization problem \eqref{MYR_min_problem}. The iterative formula of the proposed method for solving problem \eqref{MYR_min_problem} is given by
\begin{equation}\label{iter_nonsmooth}
x_{k+1} = x_k + \alpha_k d_k, 
\end{equation}
where $\alpha_k$ is the step size and $d_k$ the search direction. Generalizing \eqref{iter_seq_smooth} to the nonsmooth, inexact-gradient setting gives
\begin{equation}\label{dir_CG_nonsmooth}
d_{k+1} =
\begin{cases}
-\tilde g_0, & k=0,\\
-\tilde g_{k+1} + \beta_k d_k, & k \ge 1,
\end{cases} 
\end{equation}
and, generalizing the spectral direction \eqref{SCG_dir_smooth},
\begin{equation}\label{dir_nonsmooth_scg}
d_{k+1} =
\begin{cases}
-\tilde g_0, & k=0,\\
-\theta_{k+1}\tilde g_{k+1} + \beta_k d_k, & k \ge 1,
\end{cases} 
\end{equation}
where $\theta_{k+1}$ is the spectral parameter and $\beta_k$ the conjugate parameter. Note that \eqref{dir_nonsmooth_scg} reduces to \eqref{dir_CG_nonsmooth} when $\theta_{k+1}=1$. Since the spectral and conjugate-gradient parameters in \eqref{dir_nonsmooth_scg} can be suitably chosen, the proposed search direction can effectively incorporate the complementary advantages of both spectral and conjugate-gradient approaches.


To proceed to the next iteration by \eqref{iter_nonsmooth}, the direction $d_{k}$ is calculated by \eqref{dir_nonsmooth_scg} in which a conjugate parameter $\beta_k$ and a spectral parameter $\theta_{k+1}$ are required, and then a step length $\alpha_k$ is chosen such that the Wolfe-type conditions are satisfied. Next, we formulate these components in turn and then describe the resulting algorithm.

\subsection{The conjugate parameter $\beta_k$}

We adopt a Dai--Kou-type parameter built from the previous direction $d_k$ and the gradient difference $y_k = \tilde g_{k+1} - \tilde g_k$:
\begin{equation}\label{hat_betak}
\hat\beta_k = \frac{\tilde g_{k+1}^\top y_k}{d_k^\top y_k} - \norm{y_k}^2\frac{d_k^\top \tilde g_{k+1}}{(d_k^\top y_k)^2}, 
\end{equation}
well-defined whenever $d_k^\top y_k \neq 0$. We safeguard this by requiring $|d_k^\top y_k| \ge \rho$ for a fixed tolerance $\rho > 0$ (see Algorithm \ref{algo}, Step 4). Formula \eqref{hat_betak} is the $y_k$-based specialization of the Dai-Kou family.

The parameter $\hat\beta_k$ in \eqref{hat_betak} is not itself bounded: since $d_k^\top y_k$ may be arbitrarily small in magnitude while still clearing the safeguard threshold $\rho$, $\hat\beta_k$ can be arbitrarily large, and a direction built from  $\hat\beta_k$ need not have a bounded norm. We therefore truncate $\hat\beta_k$ before using it, at no cost to the descent guarantee established in Lemma \ref{lem:descent} below. Given $d_k$ and the newly computed $\tilde g_{k+1}$, set
\begin{equation}\label{betak}
\beta_k =
\begin{cases}
\hat\beta_k, & \text{if } |\hat\beta_k|\,\norm{d_k} \le \tfrac{1}{4}\norm{\tilde g_{k+1}},\\[4pt]
\operatorname{sign}(\hat\beta_k)\,\dfrac{\norm{\tilde g_{k+1}}}{4\norm{d_k}}, & \text{otherwise.}
\end{cases} 
\end{equation}
By construction, $|\beta_k|\,\norm{d_k} \le \tfrac14 \norm{\tilde g_{k+1}}$ in \emph{both} branches of \eqref{betak}, which is the only property of $\beta_k$ used in the proof of Lemma \ref{lem:descent}.

\subsubsection{The spectral parameter $\theta_k$}
Following a quasi-Newton consistency argument, we match the spectral conjugate gradient direction with an implicit quasi-Newton step satisfying $B_{k+1}s_k=y_k$. As in the derivation of the smooth spectral conjugate gradient parameters in \cite{faramarzi2019,wang2015}, we neglect the cross term $s_k^\top \tilde g_{k+1}$, which becomes negligible as $x_k$ approaches a stationary point. In the present setting, this argument is adapted to the Moreau-Yosida regularized gradient $\tilde g_k$ in place of the exact gradient. Accordingly, we define

\begin{equation}\label{theta_QN}
\theta_{k+1}^{QN} = 1 - \norm{y_k}^2\,\frac{d_k^\top \tilde g_{k+1}}{(d_k^\top y_k)(y_k^\top \tilde g_{k+1})}, 
\end{equation}
and safeguard it into a fixed interval,
\begin{equation}\label{theta_k}
\theta_{k+1} =
\begin{cases}
\theta_{k+1}^{QN}, & \text{if } \theta_{k+1}^{QN} \in \left[\tfrac14 + \eta,\ \tau\right],\\[4pt]
1, & \text{otherwise,}
\end{cases} 
\end{equation}
for fixed constants $0 < \eta < \tfrac34$ and $\tau > \tfrac14+\eta$. Both branches of \eqref{theta_k} (satisfy $\theta_{k+1} \ge \tfrac14+\eta$ by construction, since $1 \ge \tfrac14+\eta$ whenever $\eta < \tfrac34$, this bound is used directly in Lemma \ref{lem:descent} and Theorem \ref{thm:global}.


\subsubsection{Wolfe-type line search}

Given the descent direction from \eqref{dir_nonsmooth_scg} with $\beta_k$ and $\theta_{k+1}$ from \eqref{betak} and \eqref{theta_k}, a line search is needed to choose the step length $\alpha_k$. Motivated by the Wolfe conditions of Hager and Zhang \cite{hagerzhang} for smooth unconstrained problems \eqref{wolfe_con_smooth1}-\eqref{wolfe_con_smooth2}, we use the following nonsmooth Wolfe-type conditions:
\begin{align}
F^a(x_k + \alpha_k d_k,\ \varepsilon_{k+1}) -  F^{a}(x_k,\varepsilon_k)&\le \delta\,\alpha_k\, \tilde g_k^\top d_k, \label{nonsmooth_WC1}\\
\tilde g_{k+1}^\top d_k &\ge \sigma\, \tilde g_k^\top d_k, \label{nonsmooth_WC2}
\end{align}
where $0 < \delta \le \sigma < 1$. Condition \eqref{nonsmooth_WC1} plays the role of sufficient decrease, now stated for the computable surrogate $F^a$ rather than $F$ itself, by \eqref{eps_F}, $F(x_k+\alpha_kd_k) \le F^a(x_k+\alpha_kd_k,\varepsilon_{k+1})$, so \eqref{nonsmooth_WC2} still controls the true decrease in $F$ up to the vanishing slack $\varepsilon_{k+1}$.

\subsection{The algorithm}
\label{sec:algorithm}

Combining the direction \eqref{dir_nonsmooth_scg} with the safeguarded parameters \eqref{betak}, \eqref{theta_k} and the line search \eqref{nonsmooth_WC1}-\eqref{nonsmooth_WC2} gives the following procedure.

\begin{algorithm}[H]
\caption{Wolfe-type Spectral Conjugate Gradient for Nonsmooth Optimization}
\label{algo}
\begin{algorithmic}[1]
\Statex \textbf{Initialization:} Choose $x_0\in\R^n$, $\lambda>0$, $0<\delta\le\sigma<1$, $0<\eta<\tfrac34$, $\tau>\tfrac14+\eta$, $\varepsilon_0\in(0,1)$, safeguard $\rho>0$, tolerance $\mathrm{tol}>0$.
\Statex Compute $\tilde g_0 = g^a(x_0,\varepsilon_0)$, set $d_0=-\tilde g_0$, $\theta_0=1$, $k=0$.
\Statex \textbf{Step 1.} If $\norm{\tilde g_k} \le \mathrm{tol}$, stop.
\Statex \textbf{Step 2.} Choose $\varepsilon_{k+1}\in(0,\varepsilon_k)$, and compute $\alpha_k$ via the Wolfe-type line search \eqref{nonsmooth_WC1}-\eqref{nonsmooth_WC2}.
\Statex \textbf{Step 3.} Set $x_{k+1}=x_k+\alpha_k d_k$, $\tilde g_{k+1} = g^a(x_{k+1},\varepsilon_{k+1})$, $y_k=\tilde g_{k+1}-\tilde g_k$.
\Statex \textbf{Step 4.} If $|d_k^\top y_k| < \rho$: set $\beta_k=0$, $\theta_{k+1}=1$. Otherwise compute $\hat\beta_k$ by \eqref{hat_betak}, $\theta_{k+1}^{QN}$ by \eqref{theta_QN}, and set $\beta_k$, $\theta_{k+1}$ by \eqref{betak}, \eqref{theta_k}.
\Statex \textbf{Step 5.} Set $d_{k+1} = -\theta_{k+1}\tilde g_{k+1} + \beta_k d_k$, $k \leftarrow k+1$; go to Step 1.
\end{algorithmic}
\end{algorithm}

\begin{remark}
By inequality \eqref{g_corcive}, the exact secant pair $s_k = x_{k+1}-x_k$, $y_k = g(x_{k+1})-g(x_k)$ satisfies (on setting $x=x_{k+1}$, $y=x_k$ in \eqref{g_corcive}):
\begin{equation}
s_k^\top y_k \ge \lambda\norm{y_k}^2 \ge 0. 
\end{equation}
Algorithm \ref{algo} instead forms $y_k$ from the inexact pair $\tilde g_{k+1}-\tilde g_k$, this inexact $y_k$ deviates from the exact secant $g(x_{k+1})-g(x_k)$ by an amount of order $\sqrt{\varepsilon_k}+\sqrt{\varepsilon_{k+1}}$, so $d_k^\top y_k$ can fail to be nonnegative only by a vanishing margin as $\varepsilon_k\to0$. This is precisely why Step 4 of Algorithm \ref{algo} safeguards against $|d_k^\top y_k|<\rho$ rather than assuming $s_k^\top y_k\ge0$ outright: the safeguard is expected to be inactive except along a vanishing subsequence, rather than being a generic obstruction.
\end{remark}

\subsection{Global convergence properties}

We now prove that Algorithm \ref{algo} converges globally to a solution of \eqref{MYR_min_problem}. The argument rests on three assumptions: boundedness of the initial level set, boundedness of $f$ below, and a rate condition on the inexactness sequence $\{\varepsilon_k\}$.

\begin{assumption}\label{ass:level}
The level set $L_0 = \{x\in\R^n : F(x) \le F(x_0)\}$ is bounded.
\end{assumption}

\begin{assumption}\label{ass:below}
$f$, and hence $F$, is bounded below on $\R^n$.
\end{assumption}

\begin{assumption}\label{ass:eps}
$\varepsilon_k \downarrow 0$, and $\varepsilon_k = o(\norm{d_k}^2)$.
\end{assumption}


We first show that the direction $d_k$ is a sufficient-descent direction with a norm controlled by $\norm{\tilde g_k}$. This is the key structural property from which both the step-length bound and global convergence follow.

\begin{lemma}\label{lem:descent}
For the search direction $d_k$ defined in \eqref{dir_nonsmooth_scg}, the following hold for all $k\ge0$:
\begin{align}
\tilde g_k^\top d_k &\le -\left(\theta_k - \tfrac14\right)\norm{\tilde g_k}^2, \label{des_prop}\\
\norm{d_k} &\le \left(\theta_k+\tfrac14\right)\norm{\tilde g_k}. \label{trust_reg_prop}
\end{align}
\end{lemma}

\begin{proof}
For $k=0$, $d_0=-\tilde g_0$ and $\theta_0=1$, so
$$\tilde g_0^\top d_0 = -\norm{\tilde g_0}^2 \le -(\theta_0-\tfrac14)\norm{\tilde g_0}^2$$ 

and $\norm{d_0}=\norm{\tilde g_0}\le(\theta_0+\tfrac14)\norm{\tilde g_0}$, both \eqref{des_prop} and \eqref{trust_reg_prop} hold.

For $k\ge0$, take the inner product of \eqref{dir_nonsmooth_scg} with $\tilde g_{k+1}$:
\begin{equation}\label{beta_time_scg_dirr}
\tilde g_{k+1}^\top d_{k+1} = -\theta_{k+1}\norm{\tilde g_{k+1}}^2 + \beta_k\,\tilde g_{k+1}^\top d_k. 
\end{equation}

\smallskip
\noindent\textit{Case 1: $|\hat\beta_k|\norm{d_k}\le\tfrac14\norm{\tilde g_{k+1}}$, so $\beta_k=\hat\beta_k$.} Using \eqref{hat_betak},
\begin{equation}\label{case1_beta_time_scg_dir}
\hat\beta_k\big(\tilde g_{k+1}^\top d_k\big) = \frac{(\tilde g_{k+1}^\top y_k)(\tilde g_{k+1}^\top d_k)}{d_k^\top y_k} - \norm{y_k}^2\frac{(d_k^\top \tilde g_{k+1})^2}{(d_k^\top y_k)^2}.
\end{equation}
Set $u = \tilde g_{k+1}$ and $v = \dfrac{(\tilde g_{k+1}^\top d_k)}{d_k^\top y_k}\,y_k$. Then
\[
u^\top v = \frac{(\tilde g_{k+1}^\top y_k)(\tilde g_{k+1}^\top d_k)}{d_k^\top y_k}, \qquad \norm{v}^2 = \norm{y_k}^2\frac{(\tilde g_{k+1}^\top d_k)^2}{(d_k^\top y_k)^2},
\]
which are exactly the two terms on the right of \eqref{case1_beta_time_scg_dir} (recall $d_k^\top\tilde g_{k+1}=\tilde g_{k+1}^\top d_k$). Since $\norm{u/2-v}^2\ge0$ for any $u,v$, we have the elementary inequality $u^\top v \le \tfrac14\norm{u}^2+\norm{v}^2$, i.e.,
\begin{equation}\label{second_term_bound}
\frac{(\tilde g_{k+1}^\top y_k)(\tilde g_{k+1}^\top d_k)}{d_k^\top y_k} \le \tfrac14\norm{\tilde g_{k+1}}^2 + \norm{y_k}^2\frac{(\tilde g_{k+1}^\top d_k)^2}{(d_k^\top y_k)^2}.
\end{equation}
Substituting \eqref{second_term_bound} into \eqref{case1_beta_time_scg_dir}, we obtain
\begin{equation}\label{second_term_inequality}
\hat\beta_k\big(\tilde g_{k+1}^\top d_k\big) \le \tfrac14\norm{\tilde g_{k+1}}^2. 
\end{equation}
Substituting \eqref{second_term_inequality} into \eqref{beta_time_scg_dirr} gives $\tilde g_{k+1}^\top d_{k+1} \le -\theta_{k+1}\norm{\tilde g_{k+1}}^2 + \tfrac14\norm{\tilde g_{k+1}}^2 = -(\theta_{k+1}-\tfrac14)\norm{\tilde g_{k+1}}^2$, and \eqref{theta_k} guarantees $\theta_{k+1}\ge\tfrac14+\eta$, giving \eqref{des_prop}. The triangle inequality applied to \eqref{dir_nonsmooth_scg}, together with $|\hat\beta_k|\norm{d_k}\le\tfrac14\norm{\tilde g_{k+1}}$, gives
\[
\norm{d_{k+1}} \le \theta_{k+1}\norm{\tilde g_{k+1}} + |\hat\beta_k|\norm{d_k} \le \left(\theta_{k+1}+\tfrac14\right)\norm{\tilde g_{k+1}},
\]
which is \eqref{trust_reg_prop}.

\smallskip
\noindent\textit{Case 2: $|\hat\beta_k|\norm{d_k} > \tfrac14\norm{\tilde g_{k+1}}$, so $\beta_k = \operatorname{sign}(\hat\beta_k)\norm{\tilde g_{k+1}}/(4\norm{d_k})$.} By Cauchy--Schwarz,
\[
\beta_k\big(\tilde g_{k+1}^\top d_k\big) \le |\beta_k|\,\norm{\tilde g_{k+1}}\,\norm{d_k} = \tfrac14\norm{\tilde g_{k+1}}^2,
\]
since $|\beta_k|\norm{d_k}=\tfrac14\norm{\tilde g_{k+1}}$ by the second branch of \eqref{betak}. Substituting into \eqref{case1_beta_time_scg_dir} and using $\theta_{k+1}\ge\tfrac14+\eta$ gives \eqref{des_prop} and, via the same triangle-inequality argument, \eqref{trust_reg_prop}.
\end{proof}


Next we turn the descent bound into a lower bound on the step length itself, which is the remaining ingredient needed for global convergence.

\begin{lemma}\label{lem:steplength}
Suppose Assumption \ref{ass:eps} holds and $\{x_k\}$ is generated by Algorithm \ref{algo}. Then for any $\epsilon_0>0$ such that $\norm{\tilde g_k}\ge\epsilon_0$ for all $k$ in an infinite index set, there exist $\bar k$ and $m>0$ such that
\begin{equation}\label{least_steplength}
\alpha_k \ge m \qquad \text{for all } k > \bar k. 
\end{equation}
\end{lemma}

\begin{proof}
Subtracting $\tilde g_k^\top d_k$ from both sides of \eqref{nonsmooth_WC2},
\[
(\sigma-1)\tilde g_k^\top d_k \le (\tilde g_{k+1}-\tilde g_k)^\top d_k \le \norm{\tilde g_{k+1}-\tilde g_k}\,\norm{d_k}.
\]
By \eqref{approx_g_lips}, with $\varepsilon_3=\max\{\varepsilon_k,\varepsilon_{k+1}\}$, and $x_{k+1}-x_k=\alpha_k d_k$, we obtain
\begin{equation}\label{WC2_use_lemma}
(\sigma-1)\tilde g_k^\top d_k \le \left(\frac{1}{\lambda}\alpha_k\norm{d_k} + \sqrt{\frac{8\varepsilon_3}{\lambda}}\right)\norm{d_k}. 
\end{equation}
Rearranging \eqref{WC2_use_lemma},
\[
\alpha_k \ge \frac{\lambda(\sigma-1)\tilde g_k^\top d_k}{\norm{d_k}^2} - \sqrt{\lambda}\,\frac{\sqrt{8\varepsilon_3}}{\norm{d_k}}.
\]
Since $\tilde g_k^\top d_k < 0$ (Lemma \ref{lem:descent}), $(\sigma-1)\tilde g_k^\top d_k>0$ and by Assumption \ref{ass:eps}, $\varepsilon_3/\norm{d_k}^2 \to 0$, so the second term vanishes as $k\to\infty$. Therefore, using \eqref{des_prop}-\eqref{trust_reg_prop}, we obtain
\[\alpha_{k}\ge
\frac{\lambda(\sigma-1)\tilde g_k^\top d_k}{\norm{d_k}^2} \ge \frac{\lambda(1-\sigma)\left(\theta_k-\tfrac14\right)}{\left(\theta_k+\tfrac14\right)^2} \ge \frac{\lambda(1-\sigma)\eta}{\left(\tau+\tfrac14\right)^2},
\]
where the last step uses $\tfrac14+\eta\le\theta_k\le\tau$ for all $k$. Hence, for any $m \in \left(0,\ \dfrac{\lambda(1-\sigma)\eta}{(\tau+1/4)^2}\right)$ there exists $\bar k$ such that \eqref{least_steplength} holds for all $k>\bar k$.
\end{proof}

With Lemmas \ref{lem:descent} and \ref{lem:steplength} in hand, we can now state and prove the main global convergence result of this section.

\begin{theorem}\label{thm:global}
Under Assumptions \ref{ass:level}-\ref{ass:eps}, the sequence $\{x_k\}$ generated by Algorithm \ref{algo} satisfies
\[
\lim_{k\to\infty} \norm{g(x_k)} = 0,
\]
and every accumulation point of $\{x_k\}$ solves the nonsmooth problem \eqref{nonsmooth_prob} (equivalently, \eqref{MYR_min_problem}).
\end{theorem}

\begin{proof}
We first show
\begin{equation}\label{norm_g_0}
\lim_{k\to\infty}\norm{\tilde g_k} = 0. 
\end{equation}
Assume that \eqref{norm_g_0} is not true, then there exist $\epsilon_0>0$ and $\bar k_1>0$ such that $\norm{\tilde g_k}\ge\epsilon_0~~\text{for all }k>\bar k_1.$
Applying Lemma \ref{lem:steplength} with this $\epsilon_0$ gives  $m>0$ and $\bar k_2>0$ such that $\alpha_k\ge m$ for all $k>\bar k_2.$ Set $\bar{k}=\max\{\bar k_1,\bar k_2\},$ so that both
\begin{equation}
    \|\tilde g_k\| \ge \epsilon_0~~~~\text{and}~~~~~\alpha_k \ge m
\end{equation}
hold simultaneously for all $k>\bar k.$
By \eqref{nonsmooth_WC1} and \eqref{eps_F}, $F(x_{k+1}) \le F^a(x_{k+1},\varepsilon_{k+1}) \le F^a(x_k,\varepsilon_k) + \delta\alpha_k \tilde g_k^\top d_k$, i.e.,
\[
 F^a(x_k,\varepsilon_k) - F^a(x_{k+1},\varepsilon_{k+1}) \ge -\delta\alpha_k \tilde g_k^\top d_k \ge \delta\alpha_k\left(\theta_k-\tfrac14\right)\norm{\tilde g_k}^2 \ge \delta\,m\,\eta\,\epsilon_0^2 = c_0 > 0~~\forall~k > \bar{k}.
\]
 The last inequality holds because of \eqref{des_prop} and $\theta_k\ge\tfrac14+\eta$. Summing over $k>\bar{k}$, i.e., $k=\bar{k}+1, \bar{k}+2,\ldots,N$ we obtain

\[
 F^a(x_{k+1},\varepsilon_{k+1}) - F^a(x_{N+1},\varepsilon_{N+1}) \ge (N-\bar k)\,c_0 \;\longrightarrow\; \infty \quad \text{as } N\to\infty,
\]
so $F^a(x_{N+1},\varepsilon_{N+1}) \to -\infty$. Since $F(x_{N+1}) \le F^a(x_{N+1},\varepsilon_{N+1})$ by \eqref{eps_F}, this forces $F(x_{N+1})\to-\infty$, contradicting Assumption \ref{ass:below} ($F$ bounded below). This contradiction proves \eqref{norm_g_0}.

By \eqref{eps_g}, i.e., $\norm{\tilde g_k - g(x_k)} \le \sqrt{2\varepsilon_k/\lambda}$, and $\varepsilon_k\to0$ ( Assumption \ref{ass:eps}), combined with \eqref{norm_g_0} gives
\begin{equation}\label{norm_lim_g_0}
\lim_{k\to\infty}\norm{g(x_k)} = 0. 
\end{equation}
Let $x^*$ be an accumulation point of $\{x_k\}$, with $x_k \to x^*$ along a subsequence $K$. By continuity of $g$ (see \eqref{g_lips_cond}) and \eqref{norm_lim_g_0} gives $g(x^*)=0$. By \eqref{grad_F}, $g(x^*)=0 \iff x^*=p(x^*)$, and by property (ii) of Section \ref{section_prelim} this is equivalent to $x^*$ solving \eqref{MYR_min_problem}, equivalently \eqref{nonsmooth_prob}.
\end{proof}

\subsection{Linear convergence rate under strong convexity}

In this section, we show that if $f$ is strongly convex, Algorithm \ref{algo} converges at an R-linear rate. This mirrors the rate results available for the smooth spectral CG methods this paper builds on \cite{faramarzi2019}, and for the proximal nonlinear CG framework of \cite{hamana2025} under an analogous strong-convexity hypothesis on the smooth part of the objective. The argument requires a step-length lower bound that holds from the very first iteration (rather than only eventually, as in Lemma \ref{lem:steplength}), which we establish first under a uniform, non-asymptotic strengthening of Assumption \ref{ass:eps}.

\begin{assumption}\label{ass:unif}
There is a constant $\gamma>0$ such that
\begin{equation}\label{Assum4}
\varepsilon_k \le \gamma\,\norm{d_k}^2 \qquad \text{for all } k\ge0, 
\end{equation}
with $\gamma \le \dfrac{m_0^2}{32\lambda}$, where $m_0 := \dfrac{\lambda(1-\sigma)\eta}{(\tau+1/4)^2}$ is the constant appearing in the proof of Lemma \ref{lem:steplength}.
\end{assumption}

\begin{remark}
Assumption \ref{ass:unif} strengthens Assumption \ref{ass:eps} from an asymptotic condition ($\varepsilon_k=o(\norm{d_k}^2)$, required only for $k$ large) to one required at every iteration. In particular, it implies Assumption \ref{ass:eps}. It is implementable in Step 2 of Algorithm \ref{algo} by capping the chosen $\varepsilon_{k+1}$ against $\gamma\norm{d_k}^2$ (or, since $d_{k+1}$ is not yet available when $\varepsilon_{k+1}$ is chosen, against the previous direction norm $\gamma\norm{d_k}^2$ as a computable proxy), and costs nothing beyond an explicit constant $\gamma$ fixed in advance.
\end{remark}

Under this strengthened assumption, the step-length bound of Lemma \ref{lem:steplength} can be made uniform in $k$, holding from the first iteration rather than only eventually.

\begin{lemma}\label{lem:unifstep}
Suppose Assumption \ref{ass:unif} holds. Then for every $k\ge0$ at which Algorithm \ref{algo} has not terminated (i.e., $\tilde g_k\neq0$),
\begin{equation}\label{uniform_steplength}
\alpha_k \ge m = \frac{m_0}{2}.
\end{equation}
\end{lemma}

\begin{proof}
Same as in the proof of Lemma \ref{lem:steplength}, \eqref{nonsmooth_WC2} together with \eqref{approx_g_lips} (using $\varepsilon_3=\varepsilon_k$, since $\{\varepsilon_k\}$ is decreasing) gives
\[
\alpha_k \ge \frac{\lambda(\sigma-1)\tilde g_k^\top d_k}{\norm{d_k}^2} - \frac{\sqrt{8\lambda\varepsilon_k}}{\norm{d_k}}~~~\text{for all }k.
\]
By \eqref{des_prop}, \eqref{trust_reg_prop} and $\tfrac14+\eta\le\theta_k\le\tau$, the first term is bounded below by $m_0$ for every $k$ with $\tilde g_k\neq0$, same as computed in the proof of Lemma \ref{lem:steplength}. For the second term, Assumption \ref{ass:unif} gives $\sqrt{8\lambda\varepsilon_k} \le \sqrt{8\lambda\gamma}\,\norm{d_k}$, so
\[
\frac{\sqrt{8\lambda\varepsilon_k}}{\norm{d_k}} \le \sqrt{8\lambda\gamma} \le \sqrt{8\lambda\cdot\frac{m_0^2}{32\lambda}} = \frac{m_0}{2},
\]
using $\gamma\le m_0^2/(32\lambda)$. Combining, $\alpha_k \ge m_0 - m_0/2 = m_0/2 = m$ for every such $k$.
\end{proof}


\begin{assumption}\label{ass:strongconvex}
$f$ is $\mu$-strongly convex for some $\mu>0$, i.e., $f(y) \ge f(x) + \xi^\top(y-x) + \tfrac{\mu}{2}\norm{y-x}^2$ for all $x,y\in\R^n$ and all $\xi\in\partial f(x)$.
\end{assumption}

\begin{lemma}\label{lem:strongF}
Under Assumption \ref{ass:strongconvex}, let $\nu = \dfrac{\mu}{1+\lambda\mu}$. Then:
\begin{enumerate}
\item[(a)] $p$ (see \eqref{exact_p}) is a contraction: $\norm{p(x_1)-p(x_2)} \le \dfrac{1}{1+\lambda\mu}\norm{x_1-x_2}$ for all $x_1,x_2\in\R^n$.
\item[(b)] $F$ is $\nu$-strongly convex. In particular, $F$ has a unique global minimizer $x^\star$, and $F(x)-F^\star \ge \tfrac{\nu}{2}\norm{x-x^\star}^2$ for all $x$, where $F^\star=F(x^\star)$.
\item[(c)] $\norm{g(x)}^2 \ge 2\nu\big(F(x)-F^\star\big)$ for all $x\in\R^n$.
\end{enumerate}
\end{lemma}

\begin{proof}
(a) Let $x_1,x_2\in\R^n$ and write $z_1=p(x_1)$ and $z_2=p(x_2)$. By optimality of $z_i$ in \eqref{exact_p}, $-\tfrac{1}{\lambda}(z_i-x_i) \in \partial f(z_i)$. Since $f$ is $\mu$-strongly convex, $\partial f$ is $\mu$-strongly monotone, so
\[
\Big(-\tfrac{1}{\lambda}(z_1-x_1)+\tfrac{1}{\lambda}(z_2-x_2)\Big)^\top(z_1-z_2) \ge \mu\norm{z_1-z_2}^2.
\]
Rearranging the left-hand side,
\[
\frac{1}{\lambda}(x_1-x_2)^\top(z_1-z_2) - \frac{1}{\lambda}\norm{z_1-z_2}^2 \ge \mu\norm{z_1-z_2}^2,
\]
so $\tfrac{1}{\lambda}(x_1-x_2)^\top(z_1-z_2) \ge (\mu+\tfrac1\lambda)\norm{z_1-z_2}^2$. By Cauchy-Schwarz inequality, the left side is at most $\tfrac1\lambda\norm{x_1-x_2}\norm{z_1-z_2}$, so 

$$\norm{x_1-x_2} \ge (1+\lambda\mu)\norm{z_1-z_2}~~\forall~x_1,x_2\in \mathbb{R}^n,$$
which is (a).

(b) Since $g(x)=(x-p(x))/\lambda$ (see \eqref{grad_F}), for $x_1,x_2\in\R^n$,
\[
\big(g(x_1)-g(x_2)\big)^\top(x_1-x_2) = \frac{1}{\lambda}\norm{x_1-x_2}^2 - \frac{1}{\lambda}\big(p(x_1)-p(x_2)\big)^\top(x_1-x_2).
\]
By Cauchy-Schwarz inequality and part (a), we obtain
$$\big(p(x_1)-p(x_2)\big)^\top(x_1-x_2) \le \norm{p(x_1)-p(x_2)}\norm{x_1-x_2} \le \tfrac{1}{1+\lambda\mu}\norm{x_1-x_2}^2.$$
Thus, 
\[
\big(g(x_1)-g(x_2)\big)^\top(x_1-x_2) \ge \frac{1}{\lambda}\Big(1-\frac{1}{1+\lambda\mu}\Big)\norm{x_1-x_2}^2 = \frac{\mu}{1+\lambda\mu}\norm{x_1-x_2}^2 = \nu\norm{x_1-x_2}^2,
\]
i.e., $g=\nabla F$ is $\nu$-strongly monotone, equivalently $F$ is $\nu$-strongly convex \cite{nesterov2003}. Strong convexity with $\nu>0$ gives coercivity, hence existence of a minimizer $x^\star$, which is unique by strict convexity; the stated quadratic bound is the standard strong-convexity inequality evaluated with $\nabla F(x^\star)=0$.

(c) Fix $x\in\R^n$. By $\nu$-strong convexity of $F$ applied at the base point $x$ with target $x^\star$,
\[
F(x^\star) \ge F(x) + g(x)^\top(x^\star-x) + \frac{\nu}{2}\norm{x^\star-x}^2.
\]
Rearranging and applying Cauchy-Schwarz inequality, with $t=\norm{x-x^\star}$,
\[
F(x)-F^\star \le g(x)^\top(x-x^\star) - \frac{\nu}{2}t^2 \le \norm{g(x)}\,t - \frac{\nu}{2}t^2 \le \max_{t\ge0}\Big(\norm{g(x)}\,t-\frac{\nu}{2}t^2\Big) = \frac{\norm{g(x)}^2}{2\nu},
\]
where the maximum is attained at $t=\norm{g(x)}/\nu$. This is (c).
\end{proof}

We can now combine the uniform step-length bound, the sufficient-descent property, and the strong-convexity properties of Lemma \ref{lem:strongF} into the main result of this subsection: an explicit R-linear rate for both function values and iterates.

\begin{theorem}\label{thm:linear}
Suppose Assumptions \ref{ass:below}, \ref{ass:eps}, \ref{ass:unif}, and \ref{ass:strongconvex} hold, and suppose the algorithm parameters satisfy
\begin{equation}\label{rho_par}
\rho = 1 - \delta m\eta\nu \in (0,1), 
\end{equation}
with $m$ as in Lemma \ref{lem:unifstep} and $\nu$ as in Lemma \ref{lem:strongF} (this can always be arranged by fixing $\lambda,\sigma,\eta,\tau,\mu$ and then choosing $\delta>0$ small enough in Algorithm \ref{algo}'s initialization, since $m$ and $\nu$ do not depend on $\delta$). Suppose further that
\begin{equation}\label{eps_cond}
\varepsilon_k \le \varepsilon_0 q^k \qquad \text{for all } k\ge0, \text{ for some fixed } q\in(0,\rho). 
\end{equation}
Then $F$ has a unique minimizer $x^\star$, and the iterates of Algorithm \ref{algo} satisfy
\begin{equation}\label{R_rate}
F(x_k) - F^\star \;\le\; K\rho^k, \qquad \norm{x_k-x^\star} \;\le\; \sqrt{\frac{2K}{\nu}}\,\rho^{k/2}, \qquad \forall k\ge0, 
\end{equation}
where $K = \big(F(x_0)-F^\star\big) + \dfrac{C}{\rho-q}$ and $C = \big(1+\tfrac{2\delta m\eta}{\lambda}\big)\varepsilon_0$. In particular, $\{x_k\}$ converges to $x^\star$ R-linearly.
\end{theorem}

\begin{proof}
Write $a_k = F(x_k)-F^\star$. Fix $k\ge0$ with $\tilde g_k\neq0$ (otherwise Algorithm \ref{algo} has already terminated, and $a_k=0$ from that point on). By \eqref{eps_F} applied at $x_{k+1}$ and then \eqref{nonsmooth_WC1},
\[
F(x_{k+1}) \le F^a(x_{k+1},\varepsilon_{k+1}) \le F^a(x_k,\varepsilon_k) + \delta\alpha_k\tilde g_k^\top d_k.
\]
Three further facts bound the right-hand side. By \eqref{nonsmooth_WC2} and $\theta_k\ge\tfrac14+\eta$, $\tilde g_k^\top d_k\le-\eta\norm{\tilde g_k}^2$. By Lemma \ref{lem:unifstep}, $\alpha_k\ge m$. By \eqref{nonsmooth_WC2} again, $\tilde F_k=F^a(x_k,\varepsilon_k)\le F(x_k)+\varepsilon_k$. Combining these three facts with the two displayed inequalities above,
\begin{equation}\label{R_1}
a_{k+1} \le a_k + \varepsilon_k - \delta m\eta\norm{\tilde g_k}^2. 
\end{equation}
We now bound $\norm{\tilde g_k}^2$ below in terms of $a_k$. Since $\norm{u}^2 \le 2\norm{u-v}^2+2\norm{v}^2$ for any $u,v$, taking $u=g(x_k)$, $v=\tilde g_k$ gives $\norm{\tilde g_k}^2 \ge \tfrac12\norm{g(x_k)}^2 - \norm{g(x_k)-\tilde g_k}^2$; by \eqref{eps_g}, $\norm{g(x_k)-\tilde g_k}^2 \le 2\varepsilon_k/\lambda$, and by Lemma \ref{lem:strongF}(c), $\norm{g(x_k)}^2 \ge 2\nu a_k$. Hence
\begin{equation}\label{R_2}
\norm{\tilde g_k}^2 \ge \nu a_k - \frac{2\varepsilon_k}{\lambda}.
\end{equation}
Substituting \eqref{R_2} into \eqref{R_1}, we obtain
\[
a_{k+1} \le a_k + \varepsilon_k - \delta m\eta\Big(\nu a_k - \frac{2\varepsilon_k}{\lambda}\Big) = (1-\delta m\eta\nu)\,a_k + \Big(1+\frac{2\delta m\eta}{\lambda}\Big)\varepsilon_k = \rho\, a_k + \Big(1+\frac{2\delta m\eta}{\lambda}\Big)\varepsilon_k.
\]
By \eqref{eps_cond}, $\varepsilon_k\le\varepsilon_0 q^k$, and $1+\tfrac{2\delta m\eta}{\lambda}>0$, so
\begin{equation}\label{ak+1}
a_{k+1} \le \rho\, a_k + C q^k, \qquad C=\Big(1+\frac{2\delta m\eta}{\lambda}\Big)\varepsilon_0. 
\end{equation}
Unrolling \eqref{ak+1} from $k=0$ (each step substitutes the previous bound and adds the next error term),
\[
a_k \;\le\; \rho^k a_0 \;+\; C\sum_{i=0}^{k-1}\rho^{k-1-i}q^i \;=\; \rho^k a_0 \;+\; C\rho^{k-1}\sum_{i=0}^{k-1}\Big(\frac{q}{\rho}\Big)^{i}.
\]
Since $0<q<\rho$, $\sum_{i=0}^{k-1}(q/\rho)^i < \sum_{i=0}^{\infty}(q/\rho)^i = \rho/(\rho-q)$, so $C\rho^{k-1}\sum_{i=0}^{k-1}(q/\rho)^i \le C\rho^k/(\rho-q)$. Hence
\[
a_k \le \rho^k\Big(a_0 + \frac{C}{\rho-q}\Big) = K\rho^k,
\]
which is the first bound in \eqref{R_rate} (the case $k=0$ holds trivially since $K\ge a_0$). For the second bound, Lemma \ref{lem:strongF}(b) gives $\tfrac{\nu}{2}\norm{x_k-x^\star}^2 \le F(x_k)-F^\star = a_k \le K\rho^k$, and solving for $\norm{x_k-x^\star}$ gives the stated bound.
\end{proof}

\begin{remark}
Condition \eqref{eps_cond} is stated as an upper bound rather than an equality precisely so that it can be satisfied simultaneously with Assumption \ref{ass:unif}: the adaptive choice
\[
\varepsilon_{k+1} \;=\; \min\big\{\varepsilon_0 q^{k+1},\ \gamma\norm{d_k}^2\big\}
\]
at Step 2 of Algorithm \ref{algo} satisfies $\varepsilon_{k+1}\le\varepsilon_0q^{k+1}$ and $\varepsilon_{k+1}\le\gamma\norm{d_k}^2$ simultaneously, by definition of the minimum, for every $k$ -- so both \eqref{Assum4} and \eqref{eps_cond} hold with no further argument required, and the schedule is computable online from quantities already available at Step 2.
\end{remark}

\end{document}